\documentclass[twoside]{amsart}
\usepackage{amscd}
\usepackage[all]{xy}

\usepackage{graphicx}

\usepackage{xcolor}

\newcommand{\Gr}{{\operatorname{Gr}}}

\renewcommand{\Im}{{\operatorname{Im}}}

\newcommand{\CH}{\operatorname{CH}}

\newcommand{\aut}{\operatorname{Aut}}
\DeclareMathOperator{\GDB}{GDB}

\renewcommand{\dim}{\operatorname{dim}}

\renewcommand{\Im}{\operatorname{Im}}

\newcommand{\C}{\mathbb{C}}
\renewcommand{\L}{\mathbf{L}}
\newcommand{\LL}{\mathcal{L}}
\renewcommand{\P}{\mathbb{P}}
\newcommand{\N}{\mathbb{N}}
\newcommand{\Q}{\mathbb{Q}}

\newcommand{\A}{\mathbb{A}}

 \newcommand{\sC}{\mathcal{C}}

\newcommand{\sM}{\mathcal{M}}

\newcommand{\sO}{\mathcal{O}}

\newcommand{\sB}{\mathcal{B}}

\newcommand{\sX}{\mathcal{X}}
\newcommand{\sY}{\mathcal{Y}}

 \numberwithin{equation}{section}

\theoremstyle{plain}
\newtheorem{thm}[equation]{Theorem}

\newtheorem{prop}[equation]{Proposition}
\newtheorem{lm}[equation]{Lemma}
\newtheorem{cor}[equation]{Corollary}
\newtheorem{conj}[equation]{Conjecture}

\newtheorem{convention}{Conventions}

\DeclareMathOperator{\GDCH}{GDCH}

\theoremstyle{definition}
\newtheorem{defn}[equation]{Definition}

\newtheorem{rk}[equation]{Remark}

\newtheorem{nonumbering}{Theorem}

\newtheorem{nonumberingp}{Proposition}

\thanks{\textit{2020 Mathematics Subject Classification:}  14C15, 14C25, 14C30}
\keywords{algebraic cycles, Chow groups, motive, Gushel--Mukai varieties, hyper-K\"ahler varieties}
\thanks{M.B. and R.L. are supported by ANR grant ANR-20-CE40-0023.} 

\newtheorem{nonumberingt}{Acknowledgments}

\newif\ifHideFoot
\HideFoottrue  
\HideFootfalse  

\ifHideFoot

\newcommand{\Michele}[1]{}
\newcommand{\Robert}[1]{}

\else 

\newcommand{\marg}[1]{\normalsize{{
			\color{red}\footnote{{\color{blue}#1}}}{\marginpar[\vskip
			-.25cm{\color{red}\hfill$\Rightarrow$\tiny\thefootnote}]{\vskip
				-.2cm{\color{red}$\Leftarrow$\tiny\thefootnote}}}}}
\newcommand{\Michele}[1]{\marg{(Michele) #1}}
\newcommand{\Robert}[1]{\marg{(Robert) #1}}

\fi

\begin{document}

\title{Some motivic properties of Gushel-Mukai sixfolds}

\author[M. Bolognesi]{Michele Bolognesi}
\address{Institut Montpellierain Alexander Grothendieck \\ %
Universit\'e de Montpellier \\ %
CNRS \\ %
Case Courrier 051 - Place Eug\`ene Bataillon \\ %
34095 Montpellier Cedex 5 \\ %
France}
\email{michele.bolognesi@umontpellier.fr}

\author[R. Laterveer]{Robert Laterveer}
\address{Institut de Recherche Mathématique Avanc\'ee, CNRS \\ %
Universit\'e de Strasbourg, \\ %
 7 Rue Ren\'e Descartes,\\ %
 67084 Strasbourg Cedex, France.}
\email{robert.laterveer@math.unistra.fr}

\begin{abstract} 
Gushel-Mukai sixfolds are an important class of so-called Fano-K3 varieties. In this paper we show that they admit a multiplicative Chow-Künneth decomposition modulo algebraic equivalence and that they have the Franchetta property. As side results, we show that double EPW sextics and cubes have the Franchetta property, modulo algebraic equivalence, and some vanishing results for the Chow ring of Gushel-Mukai sixfolds.
\end{abstract}
\maketitle

\section{Introduction}

Fano varieties of K3 type, sometimes dubbed as FK3, are a special class of smooth projective varieties. Roughly speaking, their
name comes from the fact that they have Hodge-theoretical properties that are very similar to those of a K3 surface. Of course, the quintessential example of a FK3 variety is the celebrated cubic fourfold, the zero locus of a
cubic polynomial in a five-dimensional projective space. More precisely:

\begin{defn}
Let $H$ be a Hodge structure of even weight $k$. The level $\lambda(H)$ of $H$ is defined as follows:

$$\lambda(H):= \mathrm{max}\{q-p|H^{p,q}\neq 0\}.$$

We say that $H$ is of K3 type if:

\begin{enumerate}
\item $\lambda(H)=2$;
\item $h^{\frac{k-2}{2},\frac{k+2}{2}}=1.$
\end{enumerate}

\end{defn}

\begin{defn}
Let $X$ be a smooth Fano variety. We define $X$ to be of K3 type (FK3) if $\dim(X)=2d$ and the Hodge numbers $h^{p,q}=0$ for all $p\neq q$, except for $h^{d-1,d+1}(X)=h^{d+1,d-1}(X)=1$
\end{defn}

We observe that if $X$ is of FK3 type, then $H^\ast(X,\C)$ contains at least one sub-Hodge structure of K3 type.
One reason that FK3 varieties are of interest is the (expected) link with hyper-K\"ahler varieties; for more on this cf. the nice recent overview \cite{Fat} and the references given there.

Let $Y$ be a smooth projective variety  over $\C$, and let $\CH^i(Y )_\Q$ denote the Chow groups
of $Y$. The intersection product defines a ring structure on the Chow ring of $Y$. The Chow ring of K3 surfaces has a particular structure.

\begin{thm}(Beauville-Voisin \cite{Vbogo})
Let S be a K3 surface. The $\Q$-subalgebra

$$R^\ast(S) := \langle \CH^1
(S), c_j(S)\rangle  \subset \CH^\ast(S)$$

injects into cohomology under the cycle class map.
\end{thm}

Eventually Beauville \cite{Beau3} has conjectured that
for projective hyper-Kähler varieties, the Chow ring should admit a multiplicative splitting. This is the case for example for K3 surfaces and for abelian varieties. More recently Shen and Vial \cite{SV} have introduced
the notion of multiplicative Chow–Künneth decomposition (MCK decomposition), which presents a more concrete realization  of Beauville's "splitting property conjecture".

The class of varieties admitting an MCK is still not completely understood. For instance, hyperelliptic curves admit an MCK but a very general curve of genus $g>2$ does not have such a decomposition. In \cite{44}, the second named author has conjectured that a FK3
should always admit such a multiplicative Chow–Künneth decomposition. This is verified for certain FK3 varieties, for instance for cubic fourfolds \cite{FLV2} and for most of the varieties \cite{40}, \cite{44}, \cite{59} on the Fatighenti--Mongardi list \cite{FM}. Cubic fourfolds, in particular, admit also a Chow-K\"unneth decomposition that displays nicely the associated K3 surface, if it exists \cite{Bu,BP,ABP}.

In this paper we go a little further along this path.  A Gushel-Mukai variety (GM for short) is a smooth intersection of a linear subspace, a quadric and the cone over the Grassmannian $\Gr(2,5)$ inside $\P^{10}$. It is well-known that GM fourfolds and GM sixfolds are FK3. Our first main result is the following.


\begin{nonumbering}[=Theorem \ref{thm:mck}]
Let $X$ be a smooth GM sixfold, then $X$ admits an MCK decomposition, modulo algebraic equivalence.
\end{nonumbering}

Of course the existence of an MCK decomposition modulo rational equivalence would imply our Theorem \ref{thm:mck}, and it seems likely that all GM varieties (of any dimension) admit an MCK decomposition modulo rational equivalence (cf. Remark \ref{problem} below for the obstacle to proving this for GM sixfolds).

Our proof of Theorem \ref{thm:mck} employs instances of the {\em Franchetta property\/}. This property is motivated by a
conjectural property of K3 surfaces (the so-called Franchetta conjecture, as proposed by O'Grady \cite{OG}, cf. also \cite{PSY}). Given  any  smooth  family  of  projective  varieties  $\mathcal{X}\to\mathcal{F}$ with smooth $\mathcal{F}$, we say  that  $\mathcal{X}\to\mathcal{F}$ satisfies  the  Franchetta  property if 
the subring of {\em generically defined cycles\/}
  \[ \GDCH^\ast_{\mathcal{F}}(X_b):= \Im\bigl(  \CH^\ast(\mathcal{X})\to \CH^\ast(X_b)\bigr) \]
  injects into cohomology, for every fibre $X_b$, $b\in\mathcal{F}$.
%

\medskip

GM sixfolds also share some Chow-theoretical properties with another class of very famous FK3 varieties, that is the Debarre--Voisin 20-fold. In fact, using methods very close to those developed in \cite{V1}, we show that for all GM sixfolds $X$, we have

\begin{equation}
\CH_0^{hom}(X)=0;\ \ \ \CH_1^{hom}(X)=0.
\end{equation}

More generally, we also have the following.

\begin{nonumberingp}[=Corollary \ref{onlyone}]
Let $X$ be a GM sixfold. Then

$$\CH^i_{hom}=0\ \ \mathrm{if}\ i\neq 4.$$
\end{nonumberingp}

Moreover, the non-vanishing Chow group is generated by the so called $\sigma$-planes, which are special 2-dimensional cycles inside $X$. 

\medskip

Finally, by imitating constructions already made for hyperelliptic curves, K3 surfaces, and cubic hypersurfaces, we introduce the \it tautological ring \rm of the $m$-th self-product of a GM sixfold as the subgroup of algebraic cycles (up to algebraic equivalence) generated by the algebraic cycles of the Grassmannian $\Gr(2,5)$ and the diagonal (see Definition \ref{tautring}). We denote by $R^*(X^m)$ the tautological ring. We show that this ring enjoys properties similar to those known for the other varieties mentioned here.

\begin{nonumbering}[=Theorem \ref{tautsubring}]
Let $X$ be a  GM sixfold. The cycle class map induces injections

$$R^*(X^m) \hookrightarrow H^*(X^m,\Q),$$

for all $m\leq 45$. 

The cycle class map induces injections

$$R^*(X^m) \hookrightarrow H^*(X^m,\Q),$$

for all $m$ if and only if $X$ is Kimura finite-dimensional.
\end{nonumbering}

As is well known, one can associate a hyper-Kähler fourfold and a hyper-Kähler sixfold to a GM sixfold (or fourfold). These are named respectively double EPW sextic and double EPW cube (see \cite{OG}, \cite{OG2}, \cite{OG3}, \cite{OG4}, \cite{IKKR} for detailed constructions). As a consequence of our results, we can show that the Franchetta property holds also for double EPW sextics and for double EPW double cubes, but only modulo algebraic equivalence.

\begin{nonumbering}[=Corollary \ref{franKahler}]

\begin{enumerate}
\item Let $Y$ be a very general double EPW sextic, and $\mathcal{M}_{EPW6}$ the moduli space of such varieties. Then we have

$$\GDB^\ast_{\mathcal{M}_{EPW6}}(Y)\hookrightarrow H^*(Y,\Q);$$

\item Let $Z$ be a very general double EPW cube, and $\mathcal{M}_{EPW3}$ the moduli space of such varieties. Then we have

$$\GDB^\ast_{\mathcal{M}_{EPW3}}(Z)\hookrightarrow H^*(Z,\Q).$$

\end{enumerate}
\end{nonumbering}

Here the ``generically defined cycles'' $\GDB^\ast_{\mathcal{M}_{EPW6}}(Y)$ and $\GDB^\ast_{\mathcal{M}_{EPW3}}(Z)$  are defined as the algebraic cycles that exist relatively over the moduli space ${\mathcal{M}_{EPW6}}$, resp. ${\mathcal{M}_{EPW3}}$, modulo algebraic equivalence (cf. Definition \ref{def:gdb} below). This adds to the existing results of families of hyper-K\"ahler varieties satisfying the Franchetta property \cite{FLV}, \cite{FLV3}. 

It is worth mentioning the recent papers \cite{FM1}, \cite{FM2}, that have appeared on arXiv some months after ours, that continue the program of studying the Chow motives and algebraic cycles on 
Gushel--Mukai varieties, with some remarkable new results.

 \vskip0.5cm

\begin{convention} In this note, the word {\sl variety\/} will refer to a reduced irreducible scheme of finite type over $\C$. A {\sl subvariety\/} is a (possibly reducible) reduced subscheme which is equidimensional. 

{\bf All cycle class groups will be with rational coefficients}: we denote by $\CH_j(Y)$ (resp. $B_j(Y)$) the Chow group of $j$-dimensional cycles on $Y$ with $\Q$-coefficients modulo rational equivalence (resp. algebraic equivalence); for $Y$ smooth of dimension $n$ the notations $\CH_j(Y)$ and $\CH^{n-j}(Y)$ (resp. $B_j(Y)$ and $B^{n-j}(Y)$) are used interchangeably. 
The notations $\CH^j_{\rm hom}(Y)$ and $\CH^j_{\rm AJ}(Y)$ will be used to indicate the subgroup of homologically trivial (resp. Abel--Jacobi trivial) cycles.
For a morphism $f\colon X\to Y$, we will write $\Gamma_f\in \CH_\ast(X\times Y)$ for the graph of $f$.

The contravariant category of Chow motives (i.e., pure motives with respect to rational equivalence as in \cite{Sc}, \cite{MNP}) will be denoted 
$\sM_{\rm rat}$.
\end{convention}

\begin{nonumberingt} RL thanks MB for an invitation to Montpellier in November 2021, where this work was initiated. Thanks to Lie Fu, Grzegorz and Michal Kapustka, Giovanni Mongardi,
Alex Perry, Laura Pertusi, Xiaolei Zhao and Shizhuo Zhang for kindly answering our questions. Thanks also to the two anonymous referees for their comments that helped to improve the paper.
\end{nonumberingt}

\section{GM varieties}

\begin{defn}[Debarre--Kuznetsov \cite{DK}] A Gushel--Mukai variety is a variety $X$ obtained as a smooth dimensionally transverse intersection
  \[ X= \operatorname{CGr}(2,V_5)\cap \P(W)\cap Q\ \ \subset \P^{10}\ ,\]
  where $ \operatorname{CGr}(2,V_5)$ is the cone over the Grassmannian (of $2$-dimensional subspaces in a fixed $5$-dimensional vector space $V_5$), and $\P(W)$ and $Q$ are a linear subspace resp. a quadric.
  
  A Gushel--Mukai variety is called {\em special\/} if $\P(W)$ contains the vertex of the cone $ \operatorname{CGr}(2,V_5)$, and {\em ordinary\/} if it is not special.
  \end{defn}
  
  \begin{rk} For simplicity, we have restricted to {\em smooth\/} Gushel--Mukai varieties (for the general set--up including singularities, cf. \cite{DK}). Gushel--Mukai varieties have dimension at most $6$. Examples of Gushel--Mukai varieties are: Clifford-general curves of genus $6$; Brill--Noether general polarized $K3$ surfaces of degree $10$ (i.e. genus $6$); smooth prime Fano threefolds of degree $10$ (i.e. genus $6$) and index $1$. 
  
  There exists an intrinsic characterization of Gushel--Mukai varieties \cite[Theorem 2.3]{DK}.
  \end{rk}

  In this note, we will mainly be interested in Gushel--Mukai sixfolds. Here, the following is known:
  
  \begin{prop}[Debarre--Kuznetsov]\label{dk} Let $X$ be a Gushel--Mukai variety of dimension $6$. Then 
   \begin{enumerate}
  \item $X$ is rational;
   \item
  $X$ is special, and hence $X$ is a double cover 
   \[ X\ \xrightarrow{2:1}\ \operatorname{Gr}(2,V_5)\ \]
   branched along an ordinary Gushel--Mukai fivefold $X_0:=   \operatorname{Gr}(2,V_5)\cap Q_0$\, ;
   
   \item
   the Hodge diamond of $X$ is
    \[ \begin{array}[c]{ccccccccccccc}
      &&&&&& 1 &&&&&&\\
      &&&&&0&&0&&&&&\\
      &&&&0&&1&&0&&&&\\
      &&&0&&0&&0&&0&&&\\
      &&0&&0&&2&&0&&0&&\\
      &0&&0&&0&&0&&0&&0\\
      0&&0&&1&&22&&1&&0&&0\\
       &0&&0&&0&&0&&0&&0\\
       &&0&&0&&2&&0&&0&&\\     
      &&&0&&0&&0&&0&&&\\
       &&&&0&&1&&0&&&&\\
        &&&&&0&&0&&&&&\\      
        &&&&&& 1 &&&&&&\\
\end{array}\]
\end{enumerate}
         
               \end{prop}
               
          \begin{proof}
                    Point (1) is \cite[Proposition 4.2]{DK}.
          Point (2) is obvious.
           Point (3) is \cite[Corollary 4.5]{DK2}. 
                    \end{proof}
               
It is also worth mentioning the semi-orthogonal decomposition of the derived category of coherent sheaves, that was given by Kuznetsov and Perry in \cite{KP}. The map $X \to \operatorname{Gr}(2,V_5)$ endows $X$ with a rank 2 vector bundle $\mathcal{U}_X$.

\begin{prop}\label{sod}
If $n=dim(X)$ then there is a semi-orthogonal decomposition

$$D^{\mathbf{b}}(X)=\langle \mathcal{A}_X, \mathcal{O}_x, \mathcal{U}_X^*,\mathcal{O}_X(1),\mathcal{U}_X(1)^*,\dots ,\mathcal{O}_X(n-3),\mathcal{U}_X^*(n-3)\rangle.$$ 

The subcategory $\mathcal{A}_X$ is defined as the right orthogonal to the rest of the semi-orthogonal decomposition, and it is a K3 category whenever $n$ is even.
\end{prop}

\subsection{Lagrangian data sets, EPW sextics, etc.}

Let $X$ be a smooth $n$-dimensional Gushel-Mukai variety. The variety $X$ is an intersection of quadrics. Let us denote by $V_6$ the 6-dimensional vector space of quadrics in the ideal of $X$. Inside $V_6$, we can also identify our space $V_5$ to the hyperplane given by the space of Plücker quadrics defining $\operatorname{CGr}(2,V_5)$. We will call the \it Plücker point \rm of $X$ the point $p_X\in \P(V_6^*)$
that corresponds to the hyperplane $V_5\subset V_6$. The wedge product on $\bigwedge^3 V_6$ induces a symplectic form
with values in $\det (V_6)$. In particular, there exists a subspace $A_X\subset \bigwedge^3 V_6$, lagrangian w.r.t. this symplectic form, associated to $X$ \cite[Thm. 2.9 and 3.6]{DK1}. The triple $(V_6,V_5,A_X)$ is called a \it Lagrangian data set \rm for $X$. If $n\geq 3$ the Lagrangian space $A_X$ has no purely decomposable vectors \cite[Thm. 3.14]{DK1}, and the vector space $A_X\cap \bigwedge^3 V_5$ has dimension $5-n$ (resp. $6-n$) if $X$ is ordinary (resp. special).

\smallskip

 As it is explained in \cite[Sect. 3]{DK1}, it is also possible to construct a Gushel-Mukai variety starting from a Lagrangian data set $(V_6,V_5,A)$, whenever  $A$ has no decomposable vectors. More precisely, Debarre and Kuznetsov, generalizing a construction due to Iliev and Manivel \cite{IMens} show that, up to isomorphism, there is a unique
ordinary (resp. special) Gushel-Mukai variety with associated Lagrangian data set $(V_6,V_5,A)$ of dimension
$5-t$ (resp. $6-t$), with $t = dim(A\cap \bigwedge^3 V_5)$.

\smallskip

Thanks to the work of O'Grady, \cite{OG2}, \cite{OG}, \cite{OG4}, \cite{OG3} one can also associate a hyper-K\"ahler manifold of dimension four to any Lagrangian space $A\subset \bigwedge^3V_5$ with no decomposable vectors. More precisely, from the Lagrangian data set $(V_6,V_5,A)$ O'Grady constructs a normal integral sextic hypersurface $Y_A\subset \P(V_6)$. Then, the Lagrangian data set defines some sublocus $Y_A^{\geq 2}$ inside $Y_A$, and the double cover $\tilde{Y}_A$ of $Y_A$ ramified along $Y_A^{\geq 2}$ is a smooth hyper-Kähler fourfold, named \it double EPW sextic \rm. Double EPW sextics form a locally complete, 20-dimensional family of hyper-K\"ahler fourfolds of K3$^{[2]}$ type (i.e. deformation equivalent to the Hilbert scheme of length 2, dimension 0, subschemes of a K3 surface).

 \subsection{Abel--Jacobi isomorphism}\label{ss:aj}
  
We finish this section by recalling a result concerning the Abel-Jacobi map from \cite{DK1}, of which we will make essential use in the body of the paper.

\smallskip

 Let $X$ be a general GM sixfold. Recall from \cite[Sect. 4.1]{DK1} that, for every 1-dimensional subspace $U_1\subset V_5$, there is a projective 3-space $\P(V_5/U_1)\subset \operatorname{Gr}(2,V_5)$. We say that a linear space $P\subset \operatorname{Gr}(2,V_5)$ is a $\sigma$-space if it is contained in $\P(V_5/U_1)$, for some $U_1\subset V_5$. Call $F_2^\sigma(X)$ the smooth, 4-dimensional subvariety of the Grassmannian $G(3,11)=G(3,\wedge^2V_5\oplus \mathbb{C})$, parametrizing projective $\sigma$-planes contained in $X$. We will often simply denote it as $F$. The universal $\sigma$-plane $\LL_2^\sigma(X):=\{(x,P)| x\in P\}\subset X\times F$ gives rise to an incidence correspondence. Moreover, as shown in \cite[Section 5.2]{DK1}, $F$ has a structure of $\P^1$-bundle over a threefold $\tilde{Y}_{A,V_5}$, which is in turn an ample divisor inside the associated double EPW sextic $\tilde{Y}_A$. 
 We recall that $\tilde{Y}_A$ is a hyper-K\"ahler fourfold and so there is the Beauville--Bogomolov form on $H^2(\tilde{Y}_A)$. 
 The above gives rise to a diagram like the following:
\[
   \xymatrix{ & \LL_2^\sigma(X) \ar[dl]_q\ar[dr]^p & & & \\
   X & & F \ar[r]^{\tilde{\sigma}} & \tilde{Y}_{A,V_5} \ar@{^{(}->}[r]^\iota & \tilde{Y}_A
  }
 \]

Since $\tilde{Y}_{A,V_5}$ is an ample divisor in $\tilde{Y}_A$, restriction induces an isomorphism
  \[ \iota^\ast\colon\ \ H^2(\tilde{Y}_A,\Q)\ \xrightarrow{\cong}\ H^2(\tilde{Y}_{A,V_5},\Q)\ .\]
 Let $\Psi_2$ denote the correspondence inducing the inverse isomorphism $ H^2(\tilde{Y}_{A,V_5},\Q)\cong H^2(\tilde{Y}_A,\Q)$ (this $\Psi_2$ exists because $ \tilde{Y}_A$ verifies the standard conjectures
 \cite{ChMa}, and so the inverse of the isomorphism 
   \[ H^2(\tilde{Y}_{A,V_5},\Q)\xrightarrow{\iota_\ast} H^4( \tilde{Y}_A,\Q)\xrightarrow{\cdot [\tilde{Y}_{A,V_5}]} H^6( \tilde{Y}_A,\Q)\cong H^2( \tilde{Y}_A,\Q)\] 
   is induced by a correspondence).

In view of the above diagram, we define a correspondence $\Gamma_0$ from $X$ to the double EPW sextic $ \tilde{Y}_A$ as
  \[ \Gamma_0:= \Psi_2\circ  \Gamma_{\tilde{\sigma}}\circ \Bigl(\Delta_{F}\cdot (p_1)^\ast(\xi) \Bigr)   \circ \Gamma_p \circ {}^t \Gamma_q \ \ \in\ \CH^{4} (X\times  \tilde{Y}_A)\ .\]
 Here $\xi$ denotes the relatively ample line bundle associated to the $\P^1$-bundle $\tilde{\sigma}$, and $p_1\colon F\times F\to F$ denotes projection to the first factor. Note that the action of $\Gamma_0$ on cohomology is equal to the composition
   \[  H^i(X) \xrightarrow{ q^\ast}\ H^i( \LL_2^\sigma(X))\xrightarrow{p_\ast}\  H^{i-4}(F)\xrightarrow{\cdot\xi} H^{i-2}(F) \xrightarrow{\tilde{\sigma}_\ast} H^{i-4}(\tilde{Y}_{A,V_5}) \xrightarrow{(\Psi_2)_\ast} H^{i-4}(\tilde{Y}_A) \ .\]

\begin{prop}{\cite{DK1}}\label{dk1} Let $X$ be a general GM sixfold.
The correspondence $\Gamma_0 $ induces an isomorphism:

$$ (\Gamma_0)_*: H^6_{}(X,\mathbb{Q})_0\ \xrightarrow{\sim}\ H^2(\tilde{Y}_A,\mathbb{Q})_0\ , $$
where $H^\ast()_0$ denotes primitive cohomology. This isomorphism is compatible with cup product on $H^6_{}(X,\mathbb{Q})$ and the Beauville--Bogomolov form on $H^2(\tilde{Y}_A,\mathbb{Q})$.
\end{prop}

\begin{proof} For $X$ general, \cite[assumption (10)]{DK1} will be satisfied. The combination of \cite[Theorem 5.19]{DK1} and \cite[Proposition 5.14]{DK1} then gives that the composition
  \[   \begin{split} H^6_{}(X,\mathbb{Z})_{00}\xrightarrow{ q^\ast}\ H^6( \LL_2^\sigma(X),\mathbb{Z})\xrightarrow{p_\ast}\  H^{2}(F,\mathbb{Z})
     \xrightarrow{(\tilde{\sigma}^\ast)^{-1}} H^{2}(\tilde{Y}_{A,V_5},\mathbb{Z})&\\ \xrightarrow{(\iota^\ast)^{-1}} H^{2}(\tilde{Y}_A,&\mathbb{Z})_0\\
  \end{split} \]
  is an isomorphism of lattices (with respect to cup product on the left side, and Beauville--Bogomolov form on the right side).
Here the notation  $H^6_{}(X,\mathbb{Z})_{00}$ refers to 
the vanishing cohomology, which coincides with the primitive cohomology $H^6_{}(X,\mathbb{Z})_0$ in view of \cite[Lemma 3.8]{DK1}. 

But after tensoring with $\Q$, the above composition is exactly the map induced by the correspondence $\Gamma_0$.
\end{proof}



\section{MCK decomposition}

\begin{defn}[Murre \cite{Mur}] Let $X$ be a smooth projective $n$-dimensional variety. We say that $X$ has a {\em CK decomposition\/} if there exists a decomposition of the diagonal
   \[ \Delta_X= \pi^0_X+ \pi^1_X+\cdots +\pi^{2n}_X\ \ \ \hbox{in}\ \CH^n(X\times X)\ ,\]
  such that the cycles $\pi^i_X$ are mutually orthogonal idempotents and $\pi^i_X$ acts on cohomology as a projector on $H^i(X,\Q)$.
  
  (NB: ``CK decomposition'' is shorthand for ``Chow--K\"unneth decomposition''.)
\end{defn}

\begin{rk} According to Murre's conjectures \cite{Mur}, \cite{J4}, any smooth projective variety should have a CK decomposition.
\end{rk}

\begin{defn}[Shen--Vial \cite{SV}] Let $X$ be a smooth projective variety of dimension $n$. Let $\Delta_X^{sm}\in \CH^{2n}(X\times X\times X)$ be the class of the small diagonal
  \[ \Delta_X^{sm}:=\bigl\{ (x,x,x)\ \vert\ x\in X\bigr\}\ \subset\ X\times X\times X\ .\]
  An {\em MCK decomposition\/} is a CK decomposition $\{\pi^i\}$ of $X$ that is {\em multiplicative\/}. This means that it satisfies
  \begin{equation}\label{vani} \pi^k_X\circ \Delta_X^{sm}\circ (\pi^i_X\times \pi^j_X)=0\ \ \ \hbox{in}\ \CH^{2n}(X^2\times X)\ \ \ \hbox{for\ all\ }i+j\not=k\ .\end{equation}
  Here $\pi^i_X\times \pi^j_X$ is by definition $(p_{13})^\ast(\pi^i_X)\cdot (p_{24})^\ast(\pi^j_X)\in \CH^{2n}(X^4)$, where $p_{rs}\colon X^2\times X^2\to X^2$ is the projection on $r$th and $s$th factors.
  
 (NB: ``MCK decomposition'' is shorthand for ``multiplicative Chow--K\"unneth decomposition''.) 
  
  \end{defn}
  
  \begin{rk}\label{rem:mck} We observe that for any CK decomposition the vanishing \eqref{vani} is always true modulo homological equivalence (i.e. equality \eqref{vani} holds in $H^{2n}(X\times X\times X)$). In fact, this is due to the fact the cup product in cohomology respects the grading.
  
Let $h(X)$ denote the Chow motive of $X$.   The small diagonal (considered as a correspondence from $X\times X$ to $X$) induces the following {\em multiplication morphism\/}
    \[ \Delta_X^{sm}\colon\ \  h(X)\otimes h(X)\ \to\ h(X)\ \ \ \hbox{in}\ \sM_{\rm rat}\ .\]
 Suppose now that $X$ has a CK decomposition
  \[ h(X)=\bigoplus_{i=0}^{2n} h^i(X)\ \ \ \hbox{in}\ \sM_{\rm rat}\ .\]
 This decomposition is by definition multiplicative if for any $i,j$ the composition
  \[ h^i(X)\otimes h^j(X)\ \to\ h(X)\otimes h(X)\ \xrightarrow{\Delta_X^{sm}}\ h(X)\ \ \ \hbox{in}\ \sM_{\rm rat}\]
  factors through $h^{i+j}(X)$.
  
Suppose $X$ has an MCK decomposition. Then, by setting
    \[ \CH^i_{(j)}(X):= (\pi_X^{2i-j})_\ast \CH^i(X) \ ,\]
    one obtains a bigraded ring structure on the Chow ring. That is, the intersection product sends $\CH^i_{(j)}(X)\otimes \CH^{i^\prime}_{(j^\prime)}(X) $ to  $\CH^{i+i^\prime}_{(j+j^\prime)}(X)$.
    
      It is reasonable to expect that for any $X$ with an MCK decomposition, one has
    \[ \CH^i_{(j)}(X)\stackrel{}{=}0\ \ \ \hbox{for}\ j<0\ ,\ \ \ \CH^i_{(0)}(X)\cap \CH^i_{hom}(X)\stackrel{}{=}0\ .\]
    This is strictly related to Murre's conjectures B and D, that were formulated for any CK decomposition \cite{Mur}.

  Having an MCK decomposition is a severely restrictive property, and it is closely related to Beauville's ``splitting property conjecture'' \cite{Beau3}. 
  Let us give a short list of examples: hyperelliptic curves have an MCK decomposition \cite[Example 8.16]{SV}, but the very general curve of genus $\ge 3$ does not \cite[Example 2.3]{FLV2}. In dimension two, a smooth quartic in $\P^3$ has an MCK decomposition, but a very general surface of degree $ \ge 7$ in $\P^3$ does (conjecturally) not admit an MCK decomposition \cite[Proposition 3.4]{FLV2}.
For a discussion in greater detail, and further examples of varieties with an MCK decomposition, the reader may check \cite[Section 8]{SV}, as well as \cite{V6}, \cite{SV2}, \cite{FTV}, \cite{37}, \cite{38}, \cite{39}, \cite{40}, \cite{44}, \cite{45}, \cite{46}, \cite{48}, \cite{55}, \cite{FLV2}, \cite{g8}, \cite{59}, \cite{60}.
   \end{rk}

\section{Franchetta property}

\begin{defn}\label{frank} Let $\sY\to \sB$ be a smooth projective morphism, where $\sY, \sB$ are smooth quasi-projective varieties. We say that $\sY\to \sB$ has the {\em Franchetta property in codimension $j$\/} if the following holds: for every $\Gamma\in \CH^j(\sY)$ such that the restriction $\Gamma\vert_{Y_b}$ is homologically trivial for all $b\in \sB$, the restriction $\Gamma\vert_{Y_b}$ is zero in $\CH^j(Y_b)$ for all $b\in \sB$.
 
 We say that $\sY\to \sB$ has the {\em Franchetta property\/} if $\sY\to \sB$ has the Franchetta property in codimension $j$ for all $j$.
 \end{defn}
 
 This property is studied in \cite{BL}, \cite{FLV}, \cite{FLV3}.
 
 \begin{defn}\label{def:gd} Given a family $\sY\to \sB$ as above, with $Y:=Y_b$ a fiber, we write
   \[ \GDCH^j_\sB(Y):=\Im\Bigl( 
  \CH^j(\sY)\to \CH^j(Y)\Bigr) \]
   for the subgroup of {\em generically defined cycles}. 
  In a context where it is clear to which family we are referring, the index $\sB$ will often be suppressed from the notation.
  \end{defn}
  
  With this notation, the Franchetta property amounts to saying that $\GDCH^\ast_\sB(Y)$ injects into cohomology, under the cycle class map.

 \subsection{Franchetta for Gushel-Mukai sixfolds.}
 
  \begin{defn}\label{univ} Let $\sC$ denote the cone over the Grassmannian $\Gr(2,5)$ and let
    \[\bar{\sB}:=\P H^0(\sC,\sO_{\sC}(2))\]
    be the linear system of quadric sections of $\mathcal{C}$. We will also denote by
  $\sB\subset\bar{\sB}$  the Zariski open subset parametrizing smooth GM sixfolds, and by $\sX \to \sB$ the universal family of GM varieties over $\sB$.
  \end{defn}
  

\begin{prop}\label{frank1} Let $\sX\to\sB$ be the universal family of GM sixfolds, as in Definition \ref{univ}.
The Franchetta property holds for $\sX\to\sB$, that is
$$\GDCH^\ast_\sB(X) \hookrightarrow H^{\ast}(X,\Q).$$
\end{prop}

\begin{proof} Let $\bar{\sX}\to\bar{\sB}$ denote the universal family of all (possibly singular and degenerate) sections.
We observe that the line bundle $\sO_{\sC}(2)$ is base point free, and so $\bar{\sX}\to\sC$ has the structure of a projective bundle. Now using the projective bundle formula (and reasoning as in \cite{PSY} and \cite{FLV}), this implies that
  \[ \GDCH^\ast_\sB(X) =\Im\Bigl(  \CH^\ast(\sC)\to \CH^\ast(X)\Bigr) \ ,\]
  for any GM sixfold. But $X$ being smooth by definition, $X$ must be contained in the punctured cone $\sC^0:=\sC\setminus\{\hbox{vertex}\}$, and so the above implies
  \[ \GDCH^\ast_\sB(X) =\Im\Bigl(  \CH^\ast(\sC^0)\to \CH^\ast(X)\Bigr) \ .\]
  The punctured cone $\sC^0$ is an $\A^1$-fibration over the Grassmannian $\Gr(2,5)$, and so the above boils down to
     \[ \GDCH^\ast_\sB(X) =\Im\Bigl(  \CH^\ast(\Gr(2,5))\to \CH^\ast(X)\Bigr) \ .\]
     The proposition is now implied by the following lemma:
  
\begin{lm}\label{coneinj}
Let $r_X:\CH^*(\Gr(2,5)) \to \CH^*(X)$ be the composition of pullback $\CH^*(\Gr(2,5)) \to \CH^*(\sC)$ and restriction $\CH^*(\sC) \to \CH^*(X)$. There is an injection into cohomology
$$Im(r_X) \hookrightarrow H^*(X,\mathbb{Q}).$$
\end{lm}

To prove the lemma, we note that the Grassmannian has trivial Chow groups, i.e. $\CH^\ast(\Gr(2,5))\cong H^\ast(\Gr(2,5),\Q)$. Moreover, the natural map
  \[ H^i(\Gr(2,5),\Q)\ \to\ H^i(X,\Q)\]
  is bijective for $i< 6$, and injective for $i=6$ \cite[Proposition 3.4(b)]{DK1}. The remaining cases $i>6$ now readily follow by Poincar\'e duality.
  \end{proof}

\subsection{Franchetta for $X\times X$ in codimension 6}

\begin{prop}\label{frankholds} Let $\sX\to\sB$ be the universal family of GM sixfolds, as in Definition \ref{univ}.
The Franchetta property holds for $X\times X$ in codimension 6, that is
$$\GDCH^6_\sB(X\times X) \hookrightarrow H^{12}(X\times X,\Q).$$
\end{prop}

\begin{proof} 
Let $\sX\subset\bar{\sX}$ be the Zariski closure, as in the proof of Proposition \ref{frank1}, and let
us consider the fiber product $\bar{\mathcal{X}}\times_\mathcal{\bar{B}}\bar{\mathcal{X}}$. This comes naturally embedded in a diagram as follows

\[
   \xymatrix{ \bar{\mathcal{X}}\times_{\bar{\mathcal{B}}}\bar{\mathcal{X}}\ar[r]^\sigma \ar[d] & \mathcal{C}\times \mathcal{C}\\
  \bar{\mathcal{B}} }
 \]
Let $\Delta_\mathcal{C} \subset \mathcal{C}\times \mathcal{C}$ denote the diagonal. The line bundle $\sO_{\sC}(2)$ is very ample and so separates 2 different points on $\sC$; this means that condition $(\ast_2)$ of \cite[Definition 2.5]{FLV3} (cf. also \cite[Definition 5.6]{FLV}) is verified. Thus, the morphism
 $\sigma$ displays $\bar{\mathcal{X}}\times_{\bar{\mathcal{B}}}\bar{\mathcal{X}}$ as a $\P^r$-bundle on $\mathcal{C}\times \mathcal{C}/ \Delta_\mathcal{C}$, and as a higher dimensional projective space bundle over $\Delta_\mathcal{C}$. Now, the \it stratified projective bundle argument \rm \cite[Proposition 2.6]{FLV3} implies there is an equality of subgroups:

\begin{equation}\label{genericeq}
\GDCH^*_\sB(X\times X)=\langle \Im\Bigl(\CH^*(\sC\times \sC) \to \CH^*(X\times X) \Bigr),\Delta_X \rangle.
\end{equation}

In particular, in codimension 6, we have

\begin{equation}\label{thiseq}
\GDCH^6_\sB(X\times X)= \mathbb{Q}[\Delta_X] + 
\bigoplus_{i+j=6}
\CH^i(\sC)\otimes \CH^j(\sC).
\end{equation}

Recall from the proof of Proposition \ref{frank1} that $X$ avoids the vertex of the cone $\sC$ and so the image of $\CH^*(\sC)\to\CH^\ast(X)$ coincides with the image of $r_X\colon\CH^\ast(\Gr(2,5))\to\CH^\ast(X)$. Thus, equality \ref{thiseq} boils down to
  \[ \GDCH^6_\sB(X\times X)= \mathbb{Q}[\Delta_X] + \Im r_X\otimes \Im r_X \ .\]
Now Lemma \ref{coneinj} (combined with the K\"unneth isomorphism in cohomology) implies that there is also an injection

$$\Im r_X \otimes \Im r_X\  \hookrightarrow \ H^*(X\times X,\mathbb{Q}\ ).$$

However, the diagonal $\Delta_X$ is linearly independent of $ \Im r_X \otimes \Im r_X $ inside $H^*(X\times X,\mathbb{Q})$. In fact, $H^{4,2}(X,\Q)$ contains transcendental cycles on which $\Delta_X$ acts of course as the identity, whereas an easy argument (cf. for instance \cite[Proof of Theorem 1.7]{Lat98})  shows that decomposable correspondences $\CH^*(X)\otimes \CH^*(X)$ act as 0 on them. Hence we conclude. 
\end{proof}


\section{Vanishing results for the Chow ring}

Let $\CH_i^{hom}$ (resp. $\CH^i_{hom}$) denote the group of dimension $i$ (resp. codimension $i$) homologically trivial algebraic cycles modulo rational equivalence. 
The first result of this paper concerns the vanishing of some of these groups for GM sixfolds.

\begin{thm}\label{chhom}
Let $X\subset \P^{10}$ be a GM sixfold, then we have

$$\CH_0^{hom}(X)=0;\ \ \ \CH_1^{hom}(X)=0.$$
\end{thm}




\begin{proof} 
First, we will prove the claimed vanishings for general GM sixfolds. In a second step, a classical spread lemma will allow us to extend this to all GM sixfolds.

\smallskip

So let us now assume $X$ is a general GM sixfold.
We are going to use the Abel--Jacobi isomorphism $(\Gamma_0)_\ast$ of Proposition \ref{dk1}, which relates (the cohomology of) the GM sixfold $X$ to (the 
cohomology of) the associated double EPW sextic $\tilde{Y}_A$. Let $h\in \CH^1(\tilde{Y}_A)$ denote an ample class.
By well-known properties of the Beauville--Bogomolov form (cf. \cite[Section 8 Remarque 2]{Beau}), the compatibility of Proposition \ref{dk1} amounts to the compatibility
  \begin{equation}\label{compat}   \bigl\langle \alpha, \beta\bigr\rangle_X =  \gamma \bigl\langle   (\Gamma_0)_\ast(\alpha), h^2\cdot  (\Gamma_0)_\ast(\beta)\bigr\rangle_{\tilde{Y}_A}\ ,\ \ \gamma\in\Q^\ast \end{equation}
  for $\alpha,\beta\in H^6(X,\Q)_0$, where $\langle\ ,\ \rangle_X$ denotes cup product on the variety $X$. (NB: alternatively, one could argue using Hodge theory as in \cite[Proof of Lemma 2.2]{V1} to get \eqref{compat}).
  
By taking a general codimension 2 linear section $g\colon S\hookrightarrow \tilde{Y}_A$ we get a map
 $$g^\ast \colon \  H^2_{tr}(\tilde{Y}_A,\Q) \to H^2_{tr}(S,\Q)\ ,$$
which is an injection by the Lefschetz hyperplane Theorem. Of course $S$ is a smooth surface. Denoting by $\Gamma^{\prime}$ the composition 
  \[ \Gamma^{\prime}:=  \Gamma_g\circ \Gamma_0\ \ \in\ \CH^{4}( X\times S)\ ,\]
  we thus find that there is an injection
  \begin{equation}\label{inj} ( \Gamma^{\prime})_\ast\colon\ H^6_{tr}(X,\mathbb{Q})\ \hookrightarrow\ H^2(S,\Q) \end{equation}
  which by \eqref{compat}  is compatible with cup-product. Applying \cite[Lemma 2.9(ii)]{FV2}, this implies that the left-inverse to the injection \eqref{inj} is given by a multiple of the transpose ${}^t \Gamma^{\prime}$.
  

Observe that all the cohomology of $X$ except for $H^6_{tr}(X,\mathbb{Q})$ is algebraic, and so the above gives an injection
  \[  \Gamma_\ast\colon\ H^\ast(X,\Q)\ \hookrightarrow\ H^2(S,\Q)\oplus \Q^r(\ast) \]
  induced by some correspondence $\Gamma$ (and with left-inverse induced by ${}^t \Gamma$).
 In other words, we obtain a split injection of homological motives
\begin{equation}\label{motinj}
\Gamma\colon\ h(X) \hookrightarrow h^2(S)(-2) \oplus \bigoplus_i \mathbb{L}(*)\ \ \hbox{in}\ \sM_{\rm hom}\ ,
\end{equation}
where the $\mathbb{L}(*)$ stand for some twisted Lefschetz motives (the precise number of twisted Lefschetz motives is irrelevant for our purposes). Let $\Delta_X\subset X\times X$ be the diagonal. Now, the injection \eqref{motinj} amounts to an equality of correspondences:

\begin{equation}\label{equacoho}
\Delta_X = c {}^t\Gamma \circ \Gamma \ \ \mathrm{in}\ \ H^{12}(X\times X,\Q),\ c\in \Q^*.
\end{equation}

Moreover, all the terms in equation \eqref{equacoho} are generically defined over $\sB$, hence by Proposition \ref{frankholds}, the same holds on the level of Chow groups:

\begin{equation}\label{equachow}
\Delta_X = c {}^t\Gamma \circ \Gamma \ \ \mathrm{in}\ \ \CH^6(X\times X),\ c\in \Q^*.
\end{equation}

The upshot is that the injection \eqref{motinj} now holds also in $\sM_{\rm rat}$:

\begin{equation}\label{ratinj}
\Gamma\colon\ h(X) \hookrightarrow h^2(S)(-2) \oplus \bigoplus_i \mathbb{L}(*)\ \ \hbox{in}\ \sM_{\rm rat}.
\end{equation}

In turn, by taking Chow groups on both sides of Eq. \eqref{ratinj}, we obtain a split injection 

\begin{equation}\label{chowinj}
\CH^j_{\rm hom}(X)\  \hookrightarrow\  \CH^{j-2}(S)_{\rm hom},\ \ \ \forall j.
\end{equation}

As the right-hand side obviously vanishes for $j\ge 5$, this proves the required vanishings, for general GM sixfolds. 

\smallskip

In the second step of the proof, it remains to extend the theorem to {\em all\/} GM sixfolds. To this end, we remark that the vanishing

$$\CH^{\rm hom}_0=\CH^{\rm hom}_1=0$$
is equivalent \cite[Theorem 1.7]{Lat98} to having a decomposition 

\begin{equation}\label{diagdeco}
\Delta_X = \gamma + \gamma'\ \ \ \mathrm{in}\ \ \ \CH^6(X\times X),
\end{equation}
where $\gamma$ is a cycle supported on $(W\times X)\cup (X\times W)$ for some codimension 2 closed subvariety $W\subset X$, and $\gamma'$ is decomposable, \it i.e. \rm $\gamma'\in \CH^*(X)\otimes \CH^*(X)$.

Let $\sX\to \sB$ be the universal family of GM sixfolds, as in Definition \ref{univ}. By the above, we have thus obtained a decomposition (\ref{diagdeco}) for $X=X_b$, where $b\in\sB$ is general. The Hilbert schemes argument as detailed by Voisin \cite[Proposition 3.7]{V0} shows that there exists a decomposition (\ref{diagdeco}) where $\gamma,\gamma^\prime$ are generically defined.
The spread lemma \cite[Lemma 3.2]{Vo} then allows to extend this to {\em all\/} GM sixfolds, i.e. the decomposition
(\ref{diagdeco}) actually holds for $X=X_b$ for all $b\in\sB$. In view of the afore-mentioned equivalence, this proves the theorem.
\end{proof}

\begin{rk}
We observe that the equality $\CH_0^{hom}(X)=0$ descends also (more) easily from the fact that $X$ is Fano, and hence rationally connected. This does not change our proof for $\CH_1^{hom}(X)=0$, though.
\end{rk}

\begin{rk}
Voisin proves in \cite[Section 2]{V1} the triviality of certain Chow groups of the Debarre--Voisin 20-fold (i.e., a Pl\"ucker hyperplane section of the Grassmannian $\Gr(3,10)$), using her method of ``spread'' and the Abel--Jacobi isomorphism with the associated Debarre--Voisin hyper-K\"ahler fourfold. The proof of Theorem \ref{chhom} given above is very close in spirit to Voisin's argument; the Debarre--Voisin hyper-K\"ahler fourfold is replaced by the double EPW sextic, and Voisin's ``spread'' argument is replaced by the Franchetta property Proposition \ref{frankholds}.

We also remark that the analogous result $\CH_1^{hom}(Y)=0$ for GM {\em fivefolds\/} $Y$ was proven in \cite{46}.
\end{rk}

\subsection{Some consequences}

\begin{cor}\label{onlyone} Let $X$ be a GM sixfold. Then
  \[ \CH^i_{hom}(X)=0\ \ \forall\ i\not=4\ .\]
  Moreover, $\CH^4(X)$ is generated by $\sigma$-planes.
  \end{cor}
  
  \begin{proof} The first statement follows from the vanishing $ \CH^6_{hom}(X)= \CH^5_{hom}(X)=0$ (Theorem \ref{chhom}) by the Bloch--Srinivas ``decomposition of the diagonal'' argument (cf. \cite[Remark 1.8.1]{Lat98} for the precise result applied here).
  
  For the second statement, we recall (cf. \eqref{chowinj} above) that there is an injection
    \[ (\Gamma_0)_\ast\colon\ \ \CH^4_{hom}(X)\ \hookrightarrow\ \CH^2_{hom}(S)\ ,\]
    with left-inverse induced by the transpose ${}^t \Gamma_0$, i.e. 
    \[ ({}^t \Gamma_0)_\ast\colon\ \  \CH^2_{hom}(S)\ \to\ \CH^4_{hom}(X) \]
    is surjective. Unravelling the definition of the correspondence $\Gamma_0$ (cf. Subsection \ref{ss:aj}), this implies in particular that the universal $\sigma$-plane $ \LL_2^\sigma(X)\in \CH^4(X\times F)$ induces a surjection
      \[ ({}^t   \LL_2^\sigma(X))_\ast\colon\ \ \CH^4_{hom}(F)\ \to\ \CH^4_{hom}(X)\ ,\]
      i.e. $  \CH^4_{hom}(X)$ is generated by $\sigma$-planes. As $H^8(X,\Q)\cong\Q^2\cong H^8(\Gr(2,5),\Q)$ is also generated by $\sigma$-planes, this proves the second statement. 
   \end{proof}

\section{MCK for $X$}

The goal of this section is to prove that all GM sixfolds have a  multiplicative Chow-K\"unneth decomposition {\em modulo algebraic equivalence\/}.
As before, we write $\sX\to\sB$ for the universal family of GM sixfolds. We denote $B^\ast(X)$ the group of algebraic cycles (with $\Q$-coefficients) modulo algebraic equivalence.
The following is the obvious adaptation of Definition \ref{def:gd} to algebraic equivalence:

 \begin{defn}\label{def:gdb} Given a family $\sY\to \sB$ as in Definition \ref{def:gd}, with $Y:=Y_b$ a fiber, we write
   \[ \GDB^j_\sB(Y):=\Im\Bigl( 
  B^j(\sY)\to B^j(Y)\Bigr) \]
   for the subgroup of {\em generically defined cycles}. 
  In a context where it is clear to which family we are referring, the index $\sB$ will often be suppressed from the notation.
  \end{defn}

As a preparatory result, we first establish a ``Franchettina property'' for the square of a GM sixfold:

\begin{prop}\label{Franchettina}(the Franchettina property) Let $X$ be a GM sixfold.
The cycle class map induces a split injection 

$$\GDB_\sB^*(X\times X) \hookrightarrow H^*(X\times X,\Q).$$
\end{prop}

\begin{proof} 
Recall that we have already proven (Equality \ref{genericeq}) that
 \[  \begin{split} \GDCH^*_\sB(X\times X)&=\langle    \Im r_X\otimes \Im r_X,\Delta_X \rangle \\
 \end{split}\]                                                                 
                                                                   
                                                                    Let us denote by $c_1\in \CH^1(X), c_2\in\CH^2(X)$  the image under $r_X$ of the first (resp, second) Chern class of the tautological bundle $\mathcal{U}_X$ on the Grassmannian $\Gr(2,5)$. Recall that $\CH^\ast(\Gr(2,5))$ is generated as a $\Q$-algebra by the first and second Chern classes of $Q$ \cite{3264}. Then the $\Q$-algebra $\langle    \Im r_X\otimes \Im r_X,\Delta_X \rangle$ is generated by $\Delta_X$ and the pull-backs of $c_1$ and $c_2$ with respect to the two projections $\pi_i:X\times X\to X$, $i=1,2$. By abuse of notation, in the rest of the proof we will denote by $c_i$ also the pull-back of $c_i$ via anyone of the two projections. This will be harmless since
                                                                    
                                                                    $$\Delta_X\cdot \pi_1^\ast (c_i) =\Delta_X\cdot \pi_2^\ast (c_i) .$$
                                                                    
We now want to show that we have an equality of $\Q$-algebras

\begin{equation}\label{deca}
\GDB^\ast_\sB(X\times X)=  \Im r_X\otimes \Im r_X  \oplus \Q[\Delta_X]\ .                                                                 
\end{equation}
                                                                    
                                                                    Im $r_x$ is generated by $c_1$ and $c_2$, and Im $r_x \otimes \mathrm{Im}\ r_x$ is generated by pull-backs of $c_1$ and $c_2$ to the product. An argument analogous to the one developed at the end of Proposition \ref{frankholds} shows that $\Q[\Delta_X]$ is independent of $\Im r_X\otimes \Im r_X$, up to algebraic equivalence.

In order to get equality (\ref{deca}), one just has to check where do the $\Delta_X\cdot c_i$ land. If all the classes $\Delta_X\cdot c_i$ live in $\mathrm{Im}\ r_x \otimes \mathrm{Im}\ r_x$ then the formula holds true.

It is not hard to see that the intersection product $\Delta_X\cdot c_1$ is contained in $ \Im r_X\otimes \Im r_X$. Indeed, the excess intersection formula \cite[Theorem 6.3]{F}, applied to the inclusion morphism of $X$ into the punctured cone $\sC^0$, gives that

  \[ \Delta_X\cdot c_1 \ \ \in\ \Im\Bigl(\CH^\ast( \sC^0\times \sC^0)\to \CH^\ast(X\times X)\Bigr)=\Im r_X\otimes\Im r_X\ ,\]   
  cf. \cite[Equation (13)]{FLV3}.                                                             

It remains to show that $\Delta_X\cdot c_2$ is also contained in  $ \Im r_X\otimes \Im r_X$, modulo algebraic equivalence. Once this will be proven, then equality (\ref{deca}) will hold, and so the combination of Lemma \ref{coneinj} and Proposition \ref{frankholds} will imply that $  \GDB^\ast_\sB(X\times X)  $ injects into cohomology. Notably, since $\mathrm{Im}\ r_x$ is injected in cohomology, the injection of $\mathrm{Im}\ r_x \otimes \mathrm{Im}\ r_x$ in cohomology is a consequence of the K\"unneth formula.

\smallskip

Let us then complete the proof by showing that $\Delta_X\cdot c_2\in \Im r_X\otimes \Im r_X$. We first observe that $\Delta_X\cdot c_2 $ is contained in  $ \Im r_X\otimes \Im r_X$, modulo {\em homological\/} equivalence. Indeed, the correspondence $\Delta_X\cdot c_2 $
 acts on $H^\ast(X,\Q)$ as cupping with $c_2$. As the primitive cohomology of $X$ is concentrated in degree $6$, the correspondence $\Delta_X\cdot c_2 $ acts as 0 on the primitive cohomology. The algebraic part of $H^\ast(X,\Q)$ can be expressed in terms of $c_1$ and $c_2$, and so there exists some $p\in \Im r_X\otimes \Im r_X$ such that $p$ and  $\Delta_X\cdot c_2 $ act in the same way on $H^\ast(X,\Q)$. Manin's identity principle (plus the K\"unneth formula in cohomology) then implies that
   \[    \Delta_X\cdot c_2 = p\ \ \hbox{in}\ H^{16}(X\times X,\Q)\ .\] 
 Now we want to upgrade to algebraic equivalence. Let $\tilde{Y}_A$ be the associated double EPW sextic, and let $S\subset Y$ be a general codimension 2 section.
 As we have seen in the proof of Theorem \ref{chhom}, there is an injection 
   \[ B^8(X \times X) \ \hookrightarrow\  B^4(S\times S) \oplus \Q^r\ ,\] 
 Passing to $B^8_{hom}(X\times X)$ a very similar inclusion holds:

\[ B^8_{hom}(X \times X) \ \hookrightarrow \  B^4_{hom}(S\times S)\ .\]

But $B^4_{hom}(S\times S)=0$ (homological and algebraic equivalence coincide for zero-cycles), and so $B^8_{hom}(X\times X)=0$. This proves that $\Delta_X\cdot c_2$
lies in $\Im r_X\otimes \Im r_X$ modulo algebraic equivalence, as claimed. This ends the proof.
\end{proof}

We are now in position to prove the main result of this section.

\begin{thm}\label{thm:mck}
Let $X$ be a smooth GM sixfold, then $X$ admits an MCK decomposition modulo algebraic equivalence.
\end{thm}

\begin{proof}
For any $i\not=6$, let $\pi^i_X\in \CH^*(X)\otimes \CH^*(X)$ be the (unique) projector on $H^i(X,\Q)$ that is {\em decomposable\/} (i.e. $\pi^i\in \CH^\ast(X)\otimes \CH^\ast(X)$). We define $\pi_X^6:= \Delta_X -\sum_{i\neq 6} \pi^i_X$. Then we have a CK decomposition

$$h(X)=h^0(X)\oplus \dots \oplus h^{12},$$

with the property that

\begin{equation}\label{triv} h^i(X)=(X,\pi^i_X,0)= \bigoplus \L(\ast),\ \ \ \forall i \neq 6.\end{equation}

Now we define a submotive of $h^6(X)$ as follows. We will call \it tautological cohomology \rm the subgroup $H^6_{taut}(X,\Q)$ of $H^6(X,\Q)$ obtained as the image of the natural map $H^6(\sC^0,\Q) \to H^6(X,\Q)$. We will also denote \it primitive cohomology \rm $H^6_{prim}(X,\Q)$ the orthogonal complement (with respect to cup product) of $H^6_{taut}(X,\Q)$ inside $H^6(X,\Q)$. As before, we can define a projector $\pi^6_{taut}$ on  $H^6_{taut}(X,\Q)$ that is decomposable, and so
  \begin{equation}\label{trivtoo} h^6_{taut}(X)=(X,\pi^6_{taut},0)= \bigoplus \L(\ast)\ .\end{equation}
  
Finally, we define a projector $\pi^6_{prim}$ on $H^6_{prim}(X,\Q)$ by setting $\pi^6_X= \pi^6_{taut} + \pi^6_{prim}$.

\smallskip

Let us now prove that the CK decomposition $\{ \pi^i_X\}$ is MCK, modulo algebraic equivalence.
We observe that there are equalities
\[\begin{split} \Gamma_{ijk}:=&\pi_X^k\circ \Delta_X^{sm} \circ (\pi^i_X \times \pi^j_X) \\
=& ({}^t\pi^i_X \times {}^t\pi^j_X \times \pi^k_X)_*\Delta_X^{sm}\\
=& (\pi^{12-i}_X \times \pi^{12-j}_X \times \pi^k_X)_*\Delta_X^{sm}\ \ \in B^*(h^i\otimes h^j\otimes h^k)\ ,\\
\end{split}   \]
where the second equality is thanks to Lieberman's lemma \cite[Lemma 2.1.3]{MNP}.
We remark moreover that this cycle $\Gamma_{ijk}$ is generically defined (with respect to $\sB$), i.e. it lies in $\GDB^\ast_\sB(X\times X\times X)$.

Let us first consider a triple $\{\pi^i_X,\pi^j_X, \pi^k_X\}\neq\{\pi^6_{prim},\pi^6_{prim},\pi^6_{prim}\}$\footnote{Remark the highly non-incidental presence of the number of the beast here}.
In view of \eqref{triv} and \eqref{trivtoo}, in this case we can write
  \[ 
    \Gamma_{ijk}\ \ \in B^*(h^i\otimes h^j\otimes h^k)\  \hookrightarrow\    \ \bigoplus B^\ast(X\times X)\ \]
    (where the sum on the right-hand side comes from \eqref{triv} or \eqref{trivtoo}, plus the fact that $B^\ast(X\times X\times \L)=B^\ast(X\times X)$). 
    Since the cycle $\Gamma_{ijk}$ and the equalities \eqref{triv} and \eqref{trivtoo} are generically defined, it follows that there is actually an injection
    \[   \Gamma_{ijk}\ \ \in B^*_\sB( h^i\otimes h^j\otimes h^k)\  \hookrightarrow\    \ \bigoplus B^\ast_\sB(X\times X)\ .\] 
     But the cycle $\Gamma_{ijk}$ is homologically trivial (Remark \ref{rem:mck}), and so an application of the Franchettina property (Proposition \ref{Franchettina}) gives the vanishing
    \[ \Gamma_{ijk}=0\ \ \hbox{in}\  B^{12}(X\times X\times X)\ .\]  
    
 It only remains to treat the ``correspondence of the beast''
   \[ \Gamma_{666}:=    \pi^6_{prim}\circ \Delta_X^{sm} \circ (\pi^6_{prim} \times \pi^6_{prim})\ \ \in\ B^{12}(X\times X\times X)\ .\]
To this end, we consider the involution $\sigma\in\aut(X)$ coming from the double cover $X\to\Gr(2,5)$. 
This involution gives rise to a splitting of the motive of $X$
  \[  h(X)= h(X)^+ \oplus h(X)^-\ \ \ \hbox{in}\ \sM_{\rm rat}\ ,\]
  where $h(X)^+$ (resp. $h(X)^-$) is defined by the projector $1/2( \Delta_X+\Gamma_\sigma)$ (resp. the projector $1/2( \Delta_X-\Gamma_\sigma)$. 
 It is readily seen that there is equality
   \[   h(X)^- = h^6_{prim}(X)\ \ \ \hbox{in}\ \sM_{\rm rat}\ \]      
  (indeed, the double cover  $X\to\Gr(2,5)$ induces isomorphisms $h(X)^+\cong h(\Gr(2,5))\cong \sum_{j\not=6} h^j(X)\oplus h^6_{taut}(X)$).
 On the other hand, the intersection product map
  \[ (\Delta_X^{sm})_\ast \colon\ \ h(X)^-\otimes h(X)^- \ \to\ h(X) \ \ \hbox{in}\ \sM \]
  factors over $h(X)^+$, and so the map
   \[  h(X)^-\otimes h(X)^- \ \xrightarrow{\Delta_X^{sm}}\ h(X) \ \to\ h(X)^-\ \ \hbox{in}\ \sM \]       
 is zero. This is equivalent to the vanishing 
   \[    \pi^6_{prim}\circ \Delta_X^{sm} \circ (\pi^6_{prim} \times \pi^6_{prim})=0\ \ \hbox{in}\ \CH^{12}(X\times X\times X)\ ,\]
  and so a fortiori the correspondence of the beast vanishes modulo algebraic equivalence:
    \[\Gamma_{666}=0\ \ \  \hbox{in}\ B^{12}(X\times X\times X)\ .\]    
    This closes the proof.  
    \end{proof}
    
\begin{rk}\label{problem} For GM {\em fivefolds\/} a stronger result is true: these varieties have an MCK decomposition {\em modulo rational equivalence\/} \cite{46}. It seems likely that GM varieties (of any dimension) admit an MCK decomposition (modulo rational equivalence); we have not been able to prove this.

Inspection of the above proof gives the following: to prove that GM sixfolds have an MCK decomposition {\em modulo rational equivalence\/},
it would suffice to prove that for a very general GM sixfold $X$ one has
  \[ \Delta_X\cdot c_2\ \ \in\ \Im r_X\otimes \Im r_X\ \ \subset\ \CH^\ast(X\times X) \ ,\] 
  where $c_2$ and $r_X$ are as above. We have proven this modulo algebraic equivalence, but we don't know how to upgrade to rational equivalence.
\end{rk}

\section{The tautological ring}

\begin{defn}\label{tautring} Let $X$ be a GM sixfold, and $m\in\N$. We define the {\em tautological ring\/} as the $\Q$-subalgebra
  \[  \begin{split} R^\ast(X^m):=& \Bigl\langle  (p_i)^\ast\Im\bigl(B^\ast(\Gr(2,5))\to B^\ast(X)\bigr),(p_{ij})^\ast\Delta_X\Bigr\rangle\\
  =& \Bigl\langle  (p_i)^\ast (c_1), (p_i)^\ast(c_2), (p_{ij})^\ast(\Delta_X)\Bigr\rangle\\
       =&   \Bigl\langle (p_i)^\ast B^1(X), (p_i)^\ast B^2(X),  (p_{ij})^\ast \Delta_X\Bigr\rangle   \ \ \ \subset\ B^\ast(X)\ .\\\
       \end{split}\]
(Here $p_i:X^m\to X$ and $p_{ij}:X^m\to X^2$ denote the various projection morphisms from $X^m$ to $X$ resp. to $X^2$.)

\end{defn}

This tautological ring is similar to tautological subrings of the Chow ring that have been studied for hyperelliptic curves in \cite{Ta2}, \cite{Ta}, for K3 surfaces in \cite{V17}, \cite{Yin}, and for cubic hypersurfaces in \cite{FLV3}.

\begin{thm}\label{tautsubring} Let $X$ be a GM sixfold. The cycle class map induces injections
  \[ R^\ast(X^m)\ \hookrightarrow\ H^\ast(X^m,\Q) \]
  for all $m\le 45$.
  
  Moreover, the cycle class map induces injections
    \[ R^\ast(X^m)\ \hookrightarrow\ H^\ast(X^m,\Q) \]
 for all $m$ if and only if $X$ is Kimura finite-dimensional \cite{Kim}.
 \end{thm}
 
 \begin{proof} First, let us introduce some notations. Let $h\in B^1(X)$ be the hyperplane section class, and let $c\in B^2(X)$ denote the class induced by $c_2(Q)\in B^2(\Gr(2,5))$ where $Q$ is the tautological bundle on the Grassmannian. Note that we have
   \[  B^1(X)=\CH^1(X)=\Q h\ ,\ \ B^2(X)=\CH^2(X)=\Q h^2 \oplus \Q c \]
   (cf. Corollary \ref{onlyone}).
  We define
   \[ o:=  {1\over 10} h^6\ \ \in\ B^6(X)\]
to be the degree 1 zero-cycle proportional to $h^{6}\in B_{0}(X)$. Note that $B_0(X)\cong\Q$, and so any point on $X$ represents the class $o$. 
 For ease of notation, let us write 
   \[  \begin{split}  o_i:= p_i^\ast o\ \ \in B^6(X^m)\ ,\\
                            h_i:= p_i^\ast h\ \ \in B^1(X^m)\ ,\\ 
                            c_i:= p_i^\ast c\ \ \in B^2(X^m)\ ,\\  
                            \end{split}\]                        
 where $p_i
	: X^m \to X$ is the projection on the $i$-th factor. Let us also write
	\[ \tau := \pi^6_{prim} \ \ \in\ 	B^6(X\times X)\]
	where $\pi^6_{prim}$ is the projector on the primitive cohomology introduced in the proof of Theorem \ref{thm:mck},
	and $\tau_{i,j} := p_{i,j}^*\tau$, where $p_{i,j} : X^m \to X\times X$ is the
	projection on the product of the $i$-th and $j$-th factors. 	

To prove Theorem~\ref{tautsubring}, we first determine (just as in 
	\cite[Lemma~2.3]{Yin}) the relations on the level of cohomology among the cycles involved:
		
	\begin{prop}\label{L:Yin'}  
		The $\Q$-subalgebra $\overline{{R}}\,^*(X^m)$ of the cohomology algebra ${{H}}^*(X^m,\Q)$ generated by $o_i, h_i, c_i, \tau_{j,k}$, $1\leq i \leq m$, $1\leq j < k \leq m$, is isomorphic to the free graded $\Q$-algebra generated by $o_i, h_i, c_i, \tau_{j,k}$, modulo the following relations:
			 	\begin{equation}\label{E:X'}
	      h_i \cdot o_i = c_i \cdot o_i=0, \quad c_i^4=0, \quad c_i^2=\lambda c_i\cdot h_i^2+\mu h_i^4,\quad 
			h_i^6 = \nu c_i\cdot h_i^4=10\,o_i\,;
			\end{equation}
			\begin{equation}\label{E:X2'}
			\tau_{i,j} \cdot o_i = 0 ,\quad \tau_{i,j} \cdot h_i = \tau_{i,j}\cdot c_i= 0, \quad \tau_{i,j} \cdot \tau_{i,j} = b_{{\mathrm{prim}}}\, o_i\cdot o_j
			\,;
			\end{equation}
			\begin{equation}\label{E:X3'}
			\tau_{i,j} \cdot \tau_{i,k} = \tau_{j,k} \cdot o_i\,;
			\end{equation}
			\begin{equation}\label{E:X4'}
			\sum_{\sigma \in \mathfrak{S}_{b_{\mathrm{prim}}+1}} \mathrm{sgn}(\sigma) \prod_{i=1}^{b_{\mathrm{prim}}+1} \tau_{i,b_{\mathrm{prim}}+1 +\sigma(i)} = 0\ . 
			\end{equation}
	Here $\lambda, \mu, \nu\in\Q$ are some constants, $b_{{\mathrm{prim}}}:=\dim H^6(X,\Q)_{prim}=22$ is the rank of the primitive part of $H^*(X, \Q)$, and $ \mathfrak{S}_{b}$ denotes the symmetric group on $b$ elements acting by permutations on the factors.
	\end{prop}
	
	\begin{proof}
	First, let us check that the above relations hold in $H^*(X^m,\Q)$.
		The relations \eqref{E:X'} take place in $X$ and are clear. The relations \eqref{E:X2'} take place in $X^2$\,: the
		relations $\tau_{i,j} \cdot o_i = 0$  and $ \tau_{i,j}
		\cdot h_i =  \tau_{i,j}\cdot c_i= 0$ follow directly from \eqref{E:X'}, 
		while the relation
		$\tau_{i,j} \cdot \tau_{i,j} = b_{{\mathrm{prim}}}\, o_i\cdot o_j $ follows directly
		from the general fact that $\deg(\Delta_X\cdot \Delta_X) = \chi (X)$, the topological Euler
		characteristic of $X$.
		The relation \eqref{E:X3'} takes place in $X^3$ and follows from the
		relations~\eqref{E:X'} and~\eqref{E:X2'} together with the fact that the cohomology algebra $H^*(X^m,\Q)$ is graded. Finally, the relation \eqref{E:X4'} takes place in $X^{2(b_{\mathrm{prim}}+1)}$ and expresses the fact that the $(b_{\mathrm{prim}}+1)$-th exterior power of $H^*_{\mathrm{prim}}(X)$ vanishes.
		
		Second, in order to check that these relations generate the relations among $o_i, h_i, \tau_{j,k}$, $1\leq i \leq m$, $1\leq j < k \leq m$, it is sufficient to check (just as in~\cite[\S 3]{Yin}) that the pairing $$\overline{{R}}\,^i(X^m) \times \overline{{R}}\,^{mn-i}(X^m) \to \overline{{R}}\,^{mn}(X^m) \simeq \Q\,o_1\cdot o_2 \cdots o_m$$ is non-degenerate. Now, the argument given in~\cite[\S 3]{Yin} adapts \emph{mutatis mutandis} to our setting (cf. \cite[Proof of Lemma 2.12]{FLV3} 
for details on the mutatis mutandis). This proves the proposition.
	\end{proof}
	
Now, let us end the proof of Theorem \ref{tautsubring}. In view of
	Lemma~\ref{L:Yin'}, it suffices to establish relations \eqref{E:X'}, \eqref{E:X2'}, \eqref{E:X3'} and \eqref{E:X4'} modulo algebraic equivalence.	
		The relations \eqref{E:X'} hold true  in ${R}^*(X)$ thanks to the Franchetta property for $X$ (Proposition \ref{frank1}): indeed, the relations \eqref{E:X'} hold in cohomology and involve only generically defined cycles, and so Proposition \ref{frank1} guarantees that these relations are true modulo rational equivalence.
	The  relations \eqref{E:X2'} hold true in ${R}^*(X^2)$ thanks to the Franchettina property for $X\times X$ (Proposition \ref{Franchettina}).
	(Note that the relations \eqref{E:X'}  in ${R}^*(X^2)$ follow also more simply from the fact that $B^{12}(X\times X) = \Q$.)
	The relation \eqref{E:X3'}  takes place in ${R}^*(X^3)$ and, given the
	relations~\eqref{E:X'} and~\eqref{E:X2'} in $R^\ast()$, it follows from the MCK decomposition modulo algebraic equivalence (Theorem \ref{thm:mck}). More precisely, Theorem \ref{thm:mck} guarantees that we have equality
   \[  \Delta_X^{\rm sm}\circ (\pi^6_{prim}\times\pi^6_{prim})=   \pi^{12}\circ \Delta_X^{\rm sm}\circ (\pi^6_{prim}\times\pi^6_{prim})  \ \ \ \hbox{in}\  B^{12}(X^3)\ ,\]
   which (using Lieberman's lemma) translates into
   \[ (\pi^6_{prim}\times \pi^6_{prim}\times\Delta_X)_\ast    \Delta_X^{\rm sm}  =   ( \pi^6_{prim}\times \pi^6_{prim}\times\pi^{12})_\ast \Delta_Y^{\rm sm}                            
                                \ \ \ \hbox{in}\ B^{12}(X^3)\ ,\]
   which implies that
   \[  \tau_{13}\cdot \tau_{23}= \tau_{12}\cdot o_3\ \ \ \hbox{in}\ B^{12}(X^3)\ .\]
	
		Finally, the relation~\eqref{E:X4'}, which takes place in ${R}^*(X^{2(b_{prim}+1)})$, holds if and only~if the motive of $X$ is Kimura finite-dimensional.
		(Indeed, let us assume the relation \eqref{E:X4'} holds true modulo algebraic equivalence. Then the Voevodsky--Voisin nilpotence theorem \cite[Theorem B-1.2]{MNP} implies that the same relation actually holds modulo rational equivalence. This means exactly that the Chow motive $\wedge^{b_{prim}+1} h_{prim}(X)$ is zero, i.e. $h_{prim}(X)$ is evenly finite-dimensional, in the language of \cite{Kim}. Since the difference between $h(X)$ and $h_{prim}(X)$ is just a sum of twisted Lefschetz motives, it follows that the Chow motive $h(X)$ is (evenly) finite-dimensional in the sense of \cite{Kim}.)
   \end{proof}

A Franchettina property holds also for $X^3$, and - as we will see in the next section - for related hyper-K\"ahler varieties. 

\begin{prop}\label{frankX3}
There is an injection

$$\GDB_\mathcal{B}^*(X^3)\hookrightarrow H^*(X^3,\Q).$$
\end{prop}

\begin{proof}
Once again we use a ``stratified projective bundle'' argument. Let $\bar{\sX}\to \bar{\mathcal{B}}$ denote the universal family of all quadric sections of the cone $\mathcal{C}$ (see Proposition \ref{frank1}). The argument is similar to that of Proposition \ref{frankholds}. Condition $(\ast_3)$ from \cite[Definition 2.5]{FLV3} holds for $H^0(\sC,\mathcal{O}_\sC(2))$, and we have a natural diagram

\[
   \xymatrix{ \mathcal{X}\times_{\bar{\mathcal{B}}}\mathcal{X}\times_{\bar{\mathcal{B}}} \mathcal{X} \ar[r] \ar[d] & \mathcal{C}\times \mathcal{C}\times \mathcal{C}\\  
\bar{\mathcal{B}}} \] 

Let $\Delta^{sm}_\mathcal{C}$ denote the small diagonal in $\mathcal{C}\times \mathcal{C}\times \mathcal{C}$.
The horizontal map displays $\mathcal{X}\times_{\bar{\sB}}\mathcal{X}\times_{\bar{\sB}} \mathcal{X}$ as a stratified projective bundle over $\mathcal{C}\times \mathcal{C}\times \mathcal{C}$, where the strata are given by the various partial diagonals. 

Since $(\ast_3)$ holds true, we can apply \cite[Proposition 2.6]{FLV3}, and obtain the following equality 

$$\GDB_\mathcal{B}^*(X^3) =\langle c_1, c_2, \Delta^{}_X\rangle,$$
Here as usual we denote by $c_1,\ c_2$ the two Chern classes of the tautological bundle of the Grassmannian, and $c_i, \Delta_X$ are considered as elements of $B^i(X^3)$ resp. of $B^6(X^3)$ via all possible pullbacks under the projections $X^3\to X$ resp. $X^3\to X^2$. Then, in order to conclude it is enough to invoke Theorem \ref{tautsubring}.
\end{proof}

\section{Families of hyper-K\"ahler varieties}

We recall the following conjecture, known as ``generalized Franchetta conjecture'':

\begin{conj} Let $\mathcal{M}$ be the moduli stack of a locally complete family of polarized hyper-K\"ahler varieties, and let $\sX\to\mathcal{M}$ be the universal family (considered as a stack). Then $\sX\to\mathcal{M}$ has the Franchetta property, i.e. there are injections
   \[ \GDCH^\ast_{\mathcal{M}}(X)\ \hookrightarrow\ H^\ast(X,\Q) \]
   for any fiber $X$. 
\end{conj}

This conjecture is studied for K3 surfaces in \cite{PSY}, and for higher-dimensional hyper-K\"ahler varieties in \cite{BL}, \cite{FLV}, \cite{FLV3}.
Here, we obtain results for new families by restricting to {\em algebraic equivalence}.

Let us recall that, associated to a given Lagrangian datum, there exists another family of hyper-K\"ahler varieties, this time of dimension 6. These are dubbed \it double EPW cubes\rm; for a complete definition of the complicated geometric construction of such objects see \cite{IKKR}. 

\smallskip


\begin{cor}\label{franKahler}

\begin{enumerate}
\item Let $Y$ be a very general double EPW sextic, and $\mathcal{M}_{EPW6}$ the moduli space of such varieties. Then we have

$$\GDB^\ast_{\mathcal{M}_{EPW6}}(Y)\hookrightarrow H^*(Y,\Q);$$

\item let $Z$ be a very general double EPW cube, and $\mathcal{M}_{EPW3}$ the moduli space of such varieties. Then we have

$$\GDB^\ast_{\mathcal{M}_{EPW3}}(Z)\hookrightarrow H^*(Z,\Q).$$

\end{enumerate}
\end{cor}

\begin{proof}
The argument is basically the same for EPW double sextics and EPW double cubes. The only difference consists in taking $X^2$ for sextics and $X^3$ for cubes. We will then develop just the second case in detail.

\smallskip

In \cite[5.4.2]{PPZ}, the authors construct  certain hyper-Kähler sixfolds (that have the same numerical invariants as double EPW cubes) as moduli spaces $\mathcal{M}_{\sigma,\nu}(\mathcal{A}_X)$ of stable objects inside the Kuznetsov component of GM sixfolds. They also observe that a very general double EPW cube would coincide with one of these sixfolds (hence endowing the very general cube with a moduli space structure), once it is proven that $X$ and the associated double EPW cube have the same period point. In the preprint \cite[Section 5 and Remark 5.7]{KKM}, the authors prove that this is indeed the case. Hence we may assume that the very general double EPW cube is a moduli space of stable objects. This means that, if $\mathcal{B}$ denotes the parameter space of GM sixfolds and $\mathcal{M}_{EPW3}$ the moduli space of double EPW cubes, then we have a commutative diagram

\[
   \xymatrix{ \mathcal{Z}_\mathcal{B}\ar[d] \ar[r] & \mathcal{Z} \ar[d] \\
 \mathcal{B} \ar[r] & \mathcal{M}_{EPW3}. }  
 \]

Here $\mathcal{Z}$ stands for the universal family of double EPW cubes, and $\mathcal{Z}_\mathcal{B}$ denotes the base change. The map on the lower level of the diagram is given by taking the moduli space of stable objects with a given Mukai vector.

\smallskip

A second important consequence of the modular nature of double EPW cubes is that, by \cite[Theorem 1.1]{FLV3}, we have a split injection

\begin{equation}
h(Z)\hookrightarrow \bigoplus_i h(X^3)(\ast)\ \ \ \hbox{in}\ \sM_{rat}\ ,
\end{equation}
that embeds the Chow motive of $Z$ inside some copies of the motive of the triple self-product of the GM sixfold (up to some twists). Moreover, this split injection is generically defined (with respect to the parameter space $\sB$ as above). It follows that there is a commutative diagram 

\[
   \xymatrix{ \GDB_\sB^*(Z) \ar@{^{(}->}[r] \ar[d] & \bigoplus_i \GDB_\sB^*(X^3) \ar[d] \\
H^*(Z,\Q)\ar@{^{(}->}[r] & \bigoplus_i H^*(X^3,\Q)\ .}  
 \]

Since the right-hand side vertical map is injective by Proposition \ref{frankX3}, this proves (2).

To prove (1), we use that the very general double EPW sextic can be described as a moduli space of stable objects in the Kuznetsov component of a GM sixfold
\cite[Proposition 5.17]{PPZ}. Then, after replacing $X^3$ by $X^2$ the same argument applies.
\end{proof}


\end{document}